\documentclass[11pt]{amsart}
\usepackage{amsfonts,amssymb,amsthm,eucal,amsmath,verbatim}
\allowdisplaybreaks
\newtheorem{thm}{Theorem}
\newtheorem{cor}[thm]{Corollary}
\newtheorem{lemma}[thm]{Lemma}

\newcommand{\R}{\mathbb{R}}

\newcommand{\E}{\mathbb{E}}

\newcommand{\s}{\mathbb{S}}
\newcommand{\I}{\mathbb{I}}

\newcommand{\sgn}{\mathop{\mathrm{sgn}}}

\newcommand{\inprod}[2]{\left\langle #1, #2 \right\rangle}

\newcommand{\ds}{\displaystyle}
\newcommand{\F}{\mathcal{F}}

\renewcommand{\P}{\mathbb{P}}

\renewcommand{\I}{\mathbb{I}}

\newcommand{\n}{\mathcal{N}}

\begin{document}

\title[Quantitative asymptotics of graphical projection pursuit]{Quantitative asymptotics \\ of graphical projection pursuit}
\author{Elizabeth Meckes}\thanks{Research supported by an American Institute
of Mathematics Five-year Fellowship.}
\address{220 Yost Hall,
  Department of Mathematics, Case Western Reserve University, 10900
  Euclid Ave., Cleveland, OH 44122.}
\email{ese3@cwru.edu}
\urladdr{http://case.edu/artsci/math/esmeckes/}

\begin{abstract}
There is a result of Diaconis and Freedman which says that, in a limiting
sense, for large collections
of high-dimensional data most one-dimensional projections of the data are 
approximately Gaussian.
This paper gives quantitative versions of that result.  For a set of 
deterministic vectors $\{x_i\}_{i=1}^n$ in $\R^d$ with $n$ and $d$ fixed, 
let $\theta\in\s^{d-1}$
be a random point of the sphere and let $\mu_n^\theta$
denote the random measure which puts mass $\frac{1}{n}$ at each of the 
points $\inprod{x_1}{\theta},\ldots,\inprod{x_n}{\theta}$.  For a fixed
bounded Lipschitz test function $f$, $Z$ a standard Gaussian random variable 
and $\sigma^2$ 
a suitable constant, an explicit bound
is derived for the quantity $\ds\P\left[\left|\int f d\mu_n^\theta-\E f(
\sigma Z)\right|>\epsilon\right]$.  A bound is also given for 
$\ds\P\left[d_{BL}(\mu_n^\theta,
\n(0,\sigma^2))>\epsilon\right]$, where $d_{BL}$ denotes the bounded-Lipschitz 
distance, which yields a lower bound on the waiting time to finding a 
non-Gaussian projection of the $\{x_i\}$ if directions are tried independently and uniformly
on $\s^{d-1}$.
\end{abstract}

\maketitle

\section{Introduction}
A foundational tool of data analysis is the projection of high-dimensional
data to a one- or two-dimensional subspace in order to visually represent the
data, and, ideally, identify underlying structure.  The question immediately
arises: which projections are interesting?  One would like to answer by 
saying that those projections which exhibit structure are interesting, however,
identifying which projections those are is not quite as straightforward as
one might think.  In particular, there are several reasons that have led to 
the idea that one should mainly look for projections which are far from
Gaussian in behavior; that Gaussian projections in fact do not generally
exhibit interesting structure.  One justification for this idea is the 
following result due to Persi Diaconis and David Freedman.
\begin{thm}[Diaconis-Freedman \cite{DiaFre}]\label{limit}
Let $x_1,\ldots,x_n$ be deterministic vectors in $\R^d$.  Suppose that
$n$, $d$ and the $x_i$ depend on a hidden index $\nu$, so that as $\nu$
tends to infinity, so do $n$ and $d$.  Suppose that there is a $\sigma^2>0$
such that, for all $\epsilon>0$,
\begin{equation}\label{limlengths}
\frac{1}{n}\Big|\left\{j\le n:\big||x_j|^2-\sigma^2d\big|>\epsilon d\right\}
\Big|\xrightarrow{\nu\to\infty}0,
\end{equation}
and suppose that
\begin{equation}\label{limorths}
\frac{1}{n^2}\Big|\left\{j,k\le n:\big|\inprod{x_j}{x_k}\big|>\epsilon d\right\}
\Big|\xrightarrow{\nu\to\infty}0.
\end{equation}
Let $\theta\in\s^{d-1}$ be distributed uniformly on the sphere, and consider 
the random measure $\mu_\nu^\theta$ which puts mass $\frac{1}{n}$ at 
each of the points $\inprod{\theta}{x_1},\ldots,\inprod{\theta}{x_n}$.  Then
as $\nu$ tends to infinity, the measures $\mu_\nu^\theta$ tend to $\n(0,\sigma^2)$
 weakly in probability.
\end{thm} 

Heuristically, Theorem \ref{limit} can be interpreted as saying that,
for a large number of high-dimensional data vectors, as long as they have
nearly the same lengths and are nearly orthogonal, most one-dimensional
projections are close to Gaussian regardless of the structure of the data.  
It is important to note that the 
conditions \eqref{limlengths} and \eqref{limorths} are not too strong; in 
particular, even though only $d$ vectors can be exactly orthogonal in $\R^d$,
the $2^d$ vertices of a unit cube centered at the origin satisfy condition 
\eqref{limorths} for
``rough orthogonality''.

A failing of the usual interpretation of Theorem \ref{limit} is that
sometimes, projections of data look nearly Gaussian for a reason; that is,
it is not always due to the central-limit type effect described by the
theorem.  Thus the question arises: is there a way to tell whether a 
Gaussian projection is interesting?  A possible answer lies in 
quantifying the theorem, and then saying that a nearly-Gaussian projection
is interesting if it is ``too close'' to Gaussian to simply be the result
of the phenomenon  
described by Theorem \ref{limit}.  By way of analogy, one has the 
Berry-Ess{\'e}en theorem stating that the rate of convergence to normal
of the sum of $n$ independent, identically distributed random variables
is of the order $\frac{1}{\sqrt{n}}$; if one has a sum of $n$ random variables
converging to Gaussian significantly faster, it must be happening for 
some reason other than just the usual central-limit theorem.  In order
to implement this idea, it is necessary (as with the Berry-Ess{\'e}en theorem)
to have a sharp quantitative version of the limit theorem in 
question.

A second motivation for proving a quantitative version of Theorem \ref{limit}
is the application to waiting times for discovering an interesting direction
on which to project data.
If a sequence
of independent random projection directions is tried until the empirical 
distribution of the projected data is more than some threshhold away
from Gaussian (in some metric on measures), and $N$ is the number of trials needed to find such
a direction, a one can easily give a lower bound for $\E N$ from the type of
quantitative theorem proved below.  

\medskip

Thus the goal of this paper is to provide a quantitative version of Theorem 
\ref{limit} in a fixed dimension $d$
and for a fixed number of data vectors $n$.
To do this, it is first necessary to replace conditions \eqref{limlengths} and
\eqref{limorths} with non-asymptotic conditions.  The conditions we will
use are the following.  
Let $\sigma^2$ be defined by  
$\frac{1}{n}\sum_{i=1}^n|x_i|^2=\sigma^2d$.  Suppose there exist $A$ and 
$B$, such that 
\begin{equation}\label{lengths}
\frac{1}{n}\sum_{i=1}^n\big|\sigma^{-2}|x_i|^2-d\big|\le A,
\end{equation} and, for all $\theta\in\s^{d-1}$,
\begin{equation}\label{orths}
\frac{1}{n}\sum_{i=1}^n\inprod{\theta}{x_i}^2\le B.
\end{equation}
For a little perspective on the restrictiveness of these conditions, 
note that, as for the conditions of Diaconis and Freedman, they hold for the 
vertices of a unit cube in $\R^d$ (with $A=0$ and $B=\frac{1}{4}$).

Under these assumptions, the following theorems hold.
\begin{thm}\label{testfcn}
Let $\{x_i\}_{i=1}^n$ be deterministic vectors in $\R^d$, subject to 
conditions \eqref{lengths} and \eqref{orths} above.  For a point $\theta\in
\s^{d-1}$, let the measure $\mu_n^\theta$ put equal mass at each of the 
points $\inprod{\theta}{x_1},\ldots,\inprod{\theta}{x_n}$.  Fix a test 
function $f:\R\to\R$ with $\|f\|_{BL}:=\|f\|_{\infty}+\sup_{x\neq y}\frac{|f(x)-
f(y)|}{|x-y|}\le1.$  Then for $Z$ a standard Gaussian random variable, 
$\theta$ chosen uniformly on the sphere, $\sigma$ defined as above, and 
$\epsilon>\max\left(\frac{2\pi\sqrt{B}}{\sqrt{d
-1}},\frac{2(A+2)}{d-1}\right),$
\begin{equation*}\begin{split}
\P\left[\left|\int f(x)d\mu_n^\theta(x)-\E f(\sigma Z)\right|>\epsilon
\right]&\le\sqrt{\frac{\pi}{2}}
e^{-\frac{(d-1)}{2^5B}\epsilon^2}.
\end{split}\end{equation*}

\end{thm}
\begin{thm}\label{main}
Let $\{x_i\}_{i=1}^n$ be deterministic vectors in $\R^d$, subject to 
conditions \eqref{lengths} and \eqref{orths} above, and again
consider the measures $\mu_n^\theta$.  If $\theta$
is chosen uniformly from $\s^{d-1}$ and $B\ge \epsilon\ge
\max\left(\left[\frac{3\cdot 2^{6}\pi B}{\sqrt{d-1}}\right]^{2/5},\frac{2(A+2
)}{d-1}\right),$ then
$$\P\left[d_{BL}(\mu_n^\theta,\n(0,\sigma^2))>\epsilon
\right]\le\frac{c_1\sqrt{B}}{
\epsilon^{3/2}}\exp\left[-\frac{
c_2(d-1)\epsilon^5}{B^2}\right],$$
with 
$c_1=48\sqrt{\pi}$, $c_2=3^{-2}2^{-16},$ and $d_{BL}$ denoting the bounded
Lipschitz distance.
\end{thm}

\medskip

{\bf Remarks:} 
\begin{enumerate}
\item It should be emphasized that the key difference between the results
proved here and the result of Diaconis and Freedman is that Theorems
\ref{testfcn} and  \ref{main}
hold for {\em fixed} dimension $d$ and number of data vectors $n$; there
are no limits in the statements of the theorems.

\item It is not necessary for $A$ and $B$ to be absolute constants; for 
the the results above to be of interest as $d\to\infty$, it is easy to see
from the statements that it is only necessary that $A=o(d)$ and 
$B=o(d)$ for Theorem \ref{testfcn} while $B=o(\sqrt{d})$ for Theorem \ref{main}.
The reader may also be wondering where the dependence on $n$ is in 
the statements above; it is built into the definition of $B$.  
Note that, by definition,
$B\ge\frac{|x_i|^2}{n}$ for each $i$; in particular,
$B\ge \frac{\sigma^2d}{n}.$  It is thus necessary that $n\to\infty$ as 
$d\to\infty$ for 
Theorem \ref{testfcn} and $n\gg \sqrt{d}$ for Theorem \ref{main}.  

\item For Theorem \ref{testfcn}, consider the special case that
$\epsilon^2=\frac{C^2
\cdot2^5B}{d-1}$ for a large constant $C$.  Then the statement becomes
\begin{equation*}\begin{split}
\P\left[\left|\int f(x)d\mu_n^\theta(x)-\E f(\sigma Z)\right|>\frac{C'}{\sqrt{
d-1}}
\right]&\le\sqrt{\frac{\pi}{2}}
e^{-C^2},
\end{split}\end{equation*}
with $C'=C\cdot4\sqrt{2B}$.
That is, roughly speaking, $\left|\int f(x)d\mu_n^\theta(x)-\E f(\sigma Z)
\right|$ is likely to be on the order of $\frac{1}{\sqrt{d}}$ or smaller.

\item
It is similarly useful to consider the following special case for 
Theorem \ref{main}.  Let $C>\frac{3}{10}$, and consider the case
$\epsilon^5=C\left(\frac{9\cdot2^{16}B^2}{d-1}\right)\log(d-1).$  Then the 
bound above becomes:
$$\P\left[d_{BL}(\mu_n^\theta,\n(0,\sigma^2))>\left(C'\frac{\log(d-1)}{
d-1}\right)^{1/5}\right]\le \frac{C''B}{(d-1)^{C-\frac{3}{10}}},$$
where $C'=9\cdot2^{16}CB^2$ and $C''=48\sqrt{\pi}C^{-3/10}$.  
Thus, roughly speaking, the 
bounded Lipschitz distance from the random measure $\mu_n^\theta$ to 
the Gaussian measure with mean zero and variance $\sigma^2$ is unlikely
to be more than a large multiple of $\left(\frac{\log(d-1)}{d-1}\right)^{1/5}$.
We make no claims of the sharpness of this result.

\end{enumerate}

Theorem \ref{main} can easily be used to give an estimate on the 
waiting time until a non-Gaussian direction is found, if directions 
are tried randomly and independently.  Specifically, we have the following
corollary.

\begin{cor}\label{wait}
Let $\theta_1,\theta_2,\theta_3,\ldots$ be a sequence of independent, uniformly 
distributed random points on $\s^{d-1}$.  Let $T_\epsilon:=\min\{j:d_{BL}(
\mu^{\theta_j}_n,\n(0,\sigma^2)>\epsilon\}.$  Then there are constants 
$c,c'$ such that 
$$\E T_\epsilon\ge \frac{c\epsilon^{3/2}}{\sqrt{B}}\exp\left(\frac{c'(d-1)
\epsilon^5}{B^2}\right).$$
\end{cor}

\section{Proofs}
This section is mainly devoted to the proofs of Theorems \ref{testfcn}
and \ref{main}, with some additional remarks following the proofs.
For the proof of Theorem \ref{testfcn}, several auxiliary results 
are needed.  The first is an abstract normal approximation for bounding
the distance of a random variable to a Gaussian random variable in 
the presence of a continuous family of exchangeable pairs.  The theorem
is an abstraction of an idea used by Stein in \cite{steintech} to 
bound the distance to Gaussian of the trace of a power of a random orthogonal
matrix.

\begin{thm}[Meckes \cite{Mec}]\label{cont}
Suppose that $(W,W_\epsilon)$ is a family of exchangeable pairs
defined on a common probability space, such that $\E W=0$ and 
$\E W^2=\sigma^2$.  Let $\F$ be a $\sigma$-algebra on this space 
with $\sigma(W)\subseteq
\F$.  Suppose there is a function $\lambda(\epsilon)$
and random variables $E, E'$ measurable with respect to $\F$, such that
\begin{enumerate}
\item \label{lin-cond}
$\frac{1}{\lambda(\epsilon)}\E\left[W_\epsilon-W\big|\F\right]
\xrightarrow[\epsilon\to0]{L_1}-W+E'.$
\item \label{quad-cond}
$\frac{1}{2\lambda(\epsilon)\sigma^2}\E\left[(W_\epsilon-W)^2
\big|\F\right]\xrightarrow[\epsilon\to0]{L_1}1+E.$
\item \label{tert-cond}
$\frac{1}{\lambda(\epsilon)}\E|W_\epsilon-W|^3\xrightarrow{
\epsilon\to0}0.$
\end{enumerate}
\medskip

Then if $Z$ is a standard normal
random variable,
$$d_{TV}(W,\sigma Z)\le\E\big|E\big|+\sqrt{\frac{\pi}{2}}\E
\big|E'\big|.$$
\end{thm}

The next result gives expressions for some mixed moments of entries of
a Haar-distributed orthogonal matrix.  See \cite{meckes-thesis}, Lemma 3.3 
and Theorem 1.6 for a detailed proof.
\begin{lemma}\label{ints1}
If $U=\left[u_{ij}\right]_{i,j=1}^d$ is an 
orthogonal matrix distributed according to Haar measure, 
then $\E\left[\prod u_{ij}^{r_{ij}}\right]$ is non-zero
if and only if $r_{i\bullet}:=\sum_{j=1}^dr_{ij}$ and $r_{\bullet j}:=
\sum_{i=1}^dr_{ij}$ are even for each $i$ and $j$.  
Second and fourth-degree moments are as follows:
\begin{enumerate}
\item For all $i,j$, $$\E\left[u_{ij}^2\right]=\frac{1}{d}.$$ 
\item For all $i,j,r,s,\alpha,\beta,\lambda,\mu$,
\begin{equation*}\begin{split}
\E\big[u_{ij}u_{rs}&u_{\alpha\beta}u_{\lambda \mu}\big]\\&=
-\frac{1}{(d-1)d(d+2)}\Big[\delta_{ir}\delta_{\alpha\lambda}\delta_{j\beta}
\delta_{s\mu}+\delta_{ir}\delta_{\alpha\lambda}\delta_{j\mu}\delta_{s\beta}+
\delta_{i\alpha}\delta_{r\lambda}\delta_{js}\delta_{\beta\mu}\\&
\qquad\qquad\qquad\qquad\qquad\quad+
\delta_{i\alpha}\delta_{r\lambda}\delta_{j\mu}\delta_{\beta s}+
\delta_{i\lambda}\delta_{r\alpha}\delta_{js}\delta_{\beta \mu}+
\delta_{i\lambda}\delta_{r\alpha}\delta_{j\beta}\delta_{s\mu}\Big]\\&\qquad
+\frac{d+1}{(d-1)d(d+2)}\Big[\delta_{ir}\delta_{\alpha\lambda}\delta_{js}
\delta_{\beta\mu}+\delta_{i\alpha}\delta_{r\lambda}\delta_{j\beta}\delta_{
s\mu}+\delta_{i\lambda}\delta_{r\alpha}\delta_{j\mu}\delta_{s\beta}
\Big].
\end{split}\end{equation*}
\label{hugemess}

\item \label{Q}For the matrix $Q=\big[q_{ij}\big]_{i,j=1}^d$ defined by
$q_{ij}:=u_{i1}u_{j2}-u_{i2}u_{j1},$ and 
for all $i,j,\ell,p$,
$$\E\left[q_{ij}q_{\ell p}
\right]=\frac{2}{d(d-1)}\big[\delta_{i\ell}\delta_{jp}-\delta_{ip}\delta_{
j\ell}\big].$$
\end{enumerate}

\end{lemma}

Finally, we will need to make use of the concentration of measure on 
the sphere, in the form of the following lemma.
\begin{lemma}[L\'evy, see \cite{MilSch}]\label{levy}
For a function $F:\s^{d-1}\to\R$, let $M_F$ denote its median with 
respect to the uniform measure (that is, for $\theta$ distributed uniformly on
$\s^{d-1}$, $\P\big[F(\theta)\le M_F\big]\ge\frac{1}{
2}$ and  $\P\big[F(\theta)\ge M_F\big]\ge\frac{1}{2}$) and let $L$ denote its 
Lipschitz constant.  Then
$$\P\left[\big|F(\theta)-M_F\big|>\epsilon\right]\le\sqrt{\frac{\pi}{2}}\exp\left[
-\frac{(d-1)\epsilon^2}{2L^2}\right].$$
\end{lemma}

With these results, it is now possible to give the proof of 
Theorem \ref{testfcn}.

\begin{proof}[Proof of Theorem \ref{testfcn}]
The proof divides into two parts.  First, an ``annealed''
 version of the theorem is proved using the infinitesimal version of Stein's
method given by Theorem \ref{cont}.
Then,
for a fixed test function $f$ and $Z$ a standard Gaussian random variable, 
the quantity $\P\left[\left|\int fd\mu_\nu^\theta-
\E f(\sigma Z)\right|>\epsilon\right]$ is bounded using the annealed theorem
together with the concentration of measure phenomenon.

\medskip

Let $\theta$ be a uniformly distributed
random point of $\s^{d-1}\subseteq\R^d$, and let $I$ be a uniformly 
distributed element of $\{1,
\ldots,n\}$, independent of $\theta$.  Consider the random variable 
$W:=\inprod{\theta}{x_I}.$
Then $\E W=0$ by symmetry and $\E W^2=\sigma^2$ by the condition$\frac{1}{n}
\sum_{i=1}^n|x_i|^2=\sigma^2d$ .  
Theorem \ref{cont} will be used to bound the total variation distance from
$W$ to $\sigma Z$, where $Z$ is a standard Gaussian random variable.

The family of exchangeable pairs needed to apply the theorem
is constructed as follows.  
For $\epsilon>0$ fixed, let 
\begin{equation*}\begin{split}
A_\epsilon&:=\begin{bmatrix}\sqrt{1-\epsilon^2}&\epsilon\\-\epsilon&
\sqrt{1-\epsilon^2}\end{bmatrix}\oplus I_{d-2}
\,=\,I_d+\begin{bmatrix}-\frac{\epsilon^2}{2}+\delta&\epsilon\\-\epsilon&
-\frac{\epsilon^2}{2}+\delta\end{bmatrix}\oplus0_{d-2},
\end{split}\end{equation*}
where $\delta=O(\epsilon^4).$ 
Let $U$ be a Haar-distributed   $d\times d$ 
random orthogonal matrix, independent of $\theta$ and $I$, and let 
$W_\epsilon=\inprod{U A_\epsilon U^T\theta}{x_I}$; the pair
$(W,W_\epsilon)$ is exchangeable for each $\epsilon>0$.

Let $K$ be the $d\times 2$ matrix made of the first two 
columns of $U$ and
$C_2=\begin{bmatrix}0&1\\-1&0\end{bmatrix}.$  Define $Q:=KC_2K^T$ (note that
this is the same $Q$ as in part \ref{Q} of Theorem \ref{ints1}).
Then by the construction of $W_\epsilon,$
\begin{equation}\label{diff5}
W_\epsilon-W=-\left(\frac{\epsilon^2}{2}+\delta\right)\inprod{KK^T\theta}{x_I}
+\epsilon\inprod{Q\theta}{x_I}.
\end{equation}

The conditions of Theorem \ref{cont} can be verified using the expressions
in Lemma \ref{ints1} as follows. 
By the lemma, $\E\big[KK^T\big]=\frac{2}{d}I$ and $\E\big[Q\big]=0,$
and so it follows from \eqref{diff5} that 
$$\E\left[W_\epsilon-W\big|W\right]=\left(-\frac{\epsilon^2}{d}+\frac{2
\delta}{n}\right)W.$$
Condition \ref{lin-cond}
of Theorem \ref{cont} is thus satisfied for $\lambda(
\epsilon)=\frac{\epsilon^2}{d}$ and $E'=0$.

For the condition \ref{quad-cond}, taking $\F=\sigma(\theta,I),$ 
Lemma \ref{ints1}, part \ref{Q} yields
\begin{equation*}\begin{split}
\frac{1}{2\lambda(\epsilon)\sigma^2}\E\left[(W_\epsilon-W)^2\big|\F\right]&
=\frac{d}{2\sigma^2}\E\left[\inprod{Q\theta}{
x_I}^2\big|\F\right]+O(\epsilon)\\&=\frac{d}{2\sigma^2}
\sum_{i,j,r,s=1}^d\E\left[q_{ij}q_{r
s}\theta_j\theta_sx_{I,i}x_{I,r}\big|\F\right]+O(\epsilon)\\&=\frac{1
}{\sigma^2(d-1)}\left[
\sum_{i,j=1}^d\theta_j^2x_{I,i}^2-\sum_{i,j=1}^d\theta_i\theta_jx_{I,i}x_{I,j}
\right]+O(\epsilon)\\
&=\frac{1}{\sigma^2(d-1)}\left[|x_I|^2-W^2\right]+O(\epsilon)\\&=
1+\frac{1}{d-1}\left[\frac{|x_I|^2}{\sigma^2}-d+1-\frac{W^2}{
\sigma^2}\right]+O(\epsilon).
\end{split}\end{equation*}
Condition \ref{quad-cond} of Theorem \ref{cont} is thus satisfied with $E=\frac{1}{d-1}
\left[\frac{|x_I|^2}{\sigma^2}-d+1-\frac{W^2}{\sigma^2}\right].$
Condition \ref{tert-cond} of the theorem is trivial by \eqref{diff5};
it follows that 
\begin{equation}\begin{split}\label{anneal}
d_{TV}(W,\sigma Z)\le\frac{1}{d-1}\E\left|\frac{|x_I|^2}{\sigma^2}-d+1-\frac{
W^2}{\sigma^2}\right|&\le\frac{1}{d-1}\left[\frac{1}{n}\sum_{i-1}^n
\left|\frac{|x_i|^2}{\sigma^2}-d\right|+2\right]\le\frac{A+2}{d-1}.
\end{split}\end{equation}
This is the annealed statement referred to at the beginning of the proof.
\medskip

We next use the concentration of measure on the sphere to show that, for a 
large measure of
 $\theta\in\s^{d-1}$, the random measure $\mu_n^\theta$ which puts mass 
$\frac{1}{n}$ at each of the $\inprod{\theta}{x_i}$
 is close to the average behavior.  To do this, we make use of L{\'e}vy's
Lemma (Lemma \ref{levy}).
Let $f:\R\to\R$ be such that $\|f\|_{BL}:=\|f\|_\infty+\sup_{x\neq y}\frac{|f(x)-
f(y)|}{|x-y|}\le1.$  Consider the 
 function $F$ defined on the sphere by 
 $$F(\theta):=\int f(x)d\mu_n^\theta(x)=\frac{1}{n}\sum_{i=1}^nf(\inprod{
\theta}{x_i}).$$

 In order to apply Lemma \ref{levy}, it is necessary to determine the Lipschitz
constant of 
 $F$.  Let $\theta, \theta'\in\s^{d-1}$.  Then, using $\|f\|_{BL}\le 1$ 
together with equation \eqref{orths},
 \begin{equation*}\begin{split}
 \big|F(\theta')-F(\theta)\big|&=\frac{1}{n}\left|\sum_{i=1}^nf(\inprod{\theta'}{x_i})-f(\inprod{\theta}{x_i})
 \right|\\&\le\frac{1}{n}\sum_{i=1}^n|\inprod{\theta'-\theta}{x_i}|\\&\le\left[\frac{1}{n}\sum_{i=1}^n
 \inprod{\theta'-\theta}{x_i}^2\right]^{1/2}\\&\le|\theta'-\theta|\sqrt{B},
 \end{split}\end{equation*}
 thus the Lipschitz constant of $F$ is bounded by $\sqrt{B}$.
  It follows from Lemma \ref{levy} that 
 $$\P\left[\left|F(\theta)-M_F\right|>\epsilon\right]\le\sqrt{
\frac{\pi}{2}}e^{-
\frac{(d-1)\epsilon^2}{2B}},$$
 where $M_F$ is the median of the function $F$.

Now, if $\theta$ is a random point of $\s^{d-1}$, then
\begin{equation}\begin{split}\label{mean-med}
\big|\E F(\theta)-M_F\big|&\le\E\big|F(\theta)-M_F\big|\\&=\int_0^\infty\P
\Big[\big|F(\theta)-M_F\big|>t\Big]dt\\&\le\int_0^\infty \sqrt{\frac{\pi}{2}}
e^{-\frac{(d-1)t^2}{2B}}dt\\&=\frac{\pi\sqrt{ B}}{2\sqrt{d-1}},
\end{split}\end{equation}
thus if $\epsilon>\frac{\pi\sqrt{B}}{\sqrt{d-1}},$ we may use concentration
about the median of $F$ to obtain concentration about the mean, with
only a loss in constants.

Note that 
$$\E F(\theta)=\E \int fd\mu_n^\theta=\E f(W)$$
for $W=\inprod{\theta}{x_I}$ as above, and so by the bound \eqref{anneal},
$$\left|\E F(\theta)-\E f(\sigma Z)\right|\le\frac{A+2}{d-1}.$$
Putting these pieces together, if $\epsilon>\max\left(\frac{2\pi\sqrt{B}}{
\sqrt{d-1}},\frac{2(A+2)}{d-1}\right)$, then
\begin{equation*}\begin{split}
\P\left[\left|\int f d\mu_n^\theta-\E f(\sigma Z)\right|
>\epsilon\right]&\le
\P\Big[\left|F(\theta)-M_F\right|>
\epsilon-|M_F-\E F(\theta)|-|\E F(\theta)-\E f(\sigma Z)|\Big]\\&
\le\P\left[\left|F(\theta)-M_F\right|>\frac{\epsilon}{4}\right]\\&
\le\sqrt{\frac{\pi}{2}}
e^{-\frac{(d-1)}{2^5B}\epsilon^2}.
\end{split}\end{equation*}

\end{proof}

\begin{proof}[Proof of Theorem \ref{main}]

The first two steps of the proof of Theorem \ref{main}
were essentially done already in the proof of Theorem
\ref{testfcn}.  From that proof, we have that if $W=\inprod{\theta}{x_I}$
for $\theta$ distributed uniformly on $\s^{d-1}$ and $I$ independent of 
$\theta$ and uniformly distributed in $\{1,\ldots,n\}$, then 
\begin{equation}\label{anneal2}
d_{TV}(W,\sigma Z)\le\frac{A+2}{d-1},
\end{equation}
for $A$ as in equation \eqref{lengths}.  Furthermore, it follows 
from equation \eqref{mean-med} in the proof of Theorem \ref{testfcn} that
for $F(\theta):=\int f d\mu_n^\theta$ and
$\epsilon>\frac{\pi\sqrt{B}}{\sqrt{d-1}},$ then
\begin{equation}\begin{split}\label{fcnconc}
\P\left[\left|F(\theta)-\E F(\theta)\right|>\epsilon\right]&\le
\P\left[\left|F(\theta)-M_F
\right|>\epsilon-\left|M_f-\E F(\theta)\right|\right]\\&\le\P\left[\left|
F(\theta)-M_F\right|>
\epsilon-\frac{\pi\sqrt{B}}{2\sqrt{d-1}}\right]\\&\le\sqrt{\frac{\pi}{2}}
e^{-\frac{(d-1)}{8B}\epsilon^2}.
\end{split}\end{equation}

In this proof, this last statement is
used together with a series of successive approximations of arbitrary 
bounded Lipschitz functions as used by Guionnet and Zeitouni \cite{GuiZei}
to obtain a bound for $\P\left[d_{BL}(\mu_n^\theta,
\n(0,\sigma^2))>\epsilon\right]$.

By definition,
$$\P\left[d_{BL}(\mu_n^\theta,\E \mu_n^\theta)>\epsilon\right]=
\P\left[\sup_{\|f\|_{BL}\le1}\left|\int fd\mu_n^\theta-\E \int fd\mu_n^\theta\right|>\epsilon\right].$$
First consider the subclass $\F_{BL,K}= \{f: \|f\|_{BL}\le1, 
supp(f)\subseteq K\}$ for a compact set $K
\subseteq\R$.  Let $\Delta=\frac{\epsilon}{4}$; for $f\in\F_{BL,K}$, 
define the approximation $f_\Delta$ as in Guionnet and Zeitouni 
\cite{GuiZei}
as follows.  Let $x_o=\inf K$ and let
$$g(x)=\begin{cases}0&x\le0;\\x&0\le x\le\Delta;\\\Delta&x\ge\Delta.\end{cases}$$
For $x\in K$, define $f_\Delta$ recursively by $f_\Delta(x_o)=0$ and 
$$f_\Delta(x)=\sum_{i=0}^{\lceil \frac{x-x_o}{\Delta}\rceil}\Big(2\I\big[f(x_o+(i+1)\Delta)
\ge f_\Delta(x_o+i\Delta)\big]-1\Big)g(x-x_o-i\Delta).$$
That is, the function $f_\Delta$ is just an approximation of $f$ by a 
function which is 
piecewise linear and has slope 1 or 
$-1$ on each of the intervals $[x_o+i\Delta,x_o+(i+1)\Delta]$.  Note that, because $\|f\|_{BL}\le 1$,
it follows that $\|f-f_\Delta\|_{\infty}\le\Delta$ and the number of distinct functions whose linear span is used
 to approximate $f$ in this way is bounded by $\frac{|K|}{\Delta}$, where $|K|$ is the diameter
 of $K$.  If $\{h_k\}_{k=1}^N$ denotes the set of functions used in the approximation $f_\Delta$ and
 $\epsilon_k$ their coefficients,
 then for $\epsilon^2>8\pi|K|\sqrt{\frac{B}{d-1}},$
 \begin{equation*}\begin{split}
 \P\left[\sup_{f\in\F_{BL,K}}\left|\int fd\mu_n^\theta-\E\int fd\mu_n^\theta\right|>\epsilon\right]&\le
 \P\left[\sup_{f\in\F_{BL,K}}\left|\int f_\Delta d\mu_n^\theta-\E\int f_\Delta d\mu_n^\theta\right|>\epsilon
 -2\Delta\right]\\&=\P\left[\sup_{f\in\F_{BL,K}}\left|\sum_{k=1}^N\epsilon_k\left(\int h_kd\mu_n^\theta-\E
 \int h_kd\mu_n^\theta\right)\right|>\frac{\epsilon}{2}\right]\\&\le\P\left[
 \sum_{k=1}^N\left|\int h_kd\mu_n^\theta-\E \int h_kd\mu_n^\theta\right|>\frac{\epsilon}{2}\right]\\&\le
 \sum_{k=1}^N\P\left[\left|\int h_kd\mu_n^\theta-\E \int h_kd\mu_n^\theta\right|>\frac{\epsilon}{2N}\right]\\
 &\le\sqrt{\frac{\pi}{2}}Ne^{-\frac{(d-1)}{8B}\left(\frac{\epsilon}{2N}\right)^2}\\&\le
 \frac{2\sqrt{2\pi}|K|}{\epsilon}e^{-\frac{(d-1)}{8B}\left(\frac{\epsilon^2}{8|K|}\right)^2}.
 \end{split}\end{equation*} 
The second-last line follows from equation \eqref{fcnconc} above, 
and the last line from the bound $N\le \frac{4|K|}{\epsilon}$.

To move to the full set $\F_{BL}:=\{f:\|f\|_{BL}\le1\}$, we make a truncation argument.  Given
$f\in\F_{BL}$ and $M>0$, define $f_M$ by
$$f_M(x)=\begin{cases}0&x\le-M-|f(-M)|;\\sgn(f(-M))\big[x+M+|f(-M)|\big]&
-M-|f(-M)|<x\le -M;\\f(x)&-M<x\le M;\\sgn(f(M))\big[|f(M)|+M-x\big]&M<x\le M+
|f(M)|;\\0&x>M+|f(M)|;\end{cases}$$
that is, $f_M$ is equal to $f$ on $[-M,M]$ and is drops off to zero 
linearly with slope 1 outside $[-M,M]$. 
Then, since $f(x)=f_M(x)$ for $x\in [-M,M]$ and $|f(x)-f_M(x)|\le 1$ for $x\notin[-M,
M]$,
\begin{equation*}\begin{split}
\left|
\int\big[f-f_M\big]d\mu_n^\theta
\right|
&\le\P\big[|\inprod{x_I}{\theta}|>M\big]\le\frac{1}{M^2}\E\big[
\inprod{x_I}{\theta}^2\big]\le\frac{B}{M^2}.
\end{split}\end{equation*}
Choosing $M$ such that $\frac{B}{M^2}=\frac{\epsilon}{4},$ it follows that for
$\epsilon^{5/2}>\frac{3\cdot 2^{6}\pi B}{\sqrt{d-1}}$,
\begin{equation*}\begin{split}
\P\left[\sup_{f\in\F_{BL}}\left|\int fd\mu_n^\theta-\E\int fd\mu_n^\theta\right|
>\epsilon\right]&\le\P\left[\sup_{f\in\F_{BL}}\left|\int f_Md\mu_n^\theta-\E\int f_M
d\mu_n^\theta\right|>\epsilon-\frac{2B}{M^2}\right]\\&\le\P\left[\sup_{g\in\F_{BL,
[-M-1,M+1]}}\left|\int gd\mu_n^\theta-\E\int g
d\mu_n^\theta\right|>\frac{\epsilon}{2}\right]\\&\le\frac{4\sqrt{2\pi}(M+1)}{
\epsilon}e^{-\frac{
(d-1)}{8B}\left(\frac{\epsilon^2}{16(M+1)}\right)^2}\\&\le
\frac{12\sqrt{2\pi B}}{
\epsilon^{3/2}}e^{-\frac{
(d-1)\epsilon^5}{9\cdot 2^{11}B^2}},
\end{split}\end{equation*}
assuming that $B\ge \epsilon$.

Recall that 
$\E \int fd\mu_n^\theta=\E f(W)$
for $W=\inprod{\theta}{x_I}$, and so by the bound \eqref{anneal2},
$$\sup_{f\in\F_{BL}}\left|\E\int fd\mu_n^\theta-\E f(\sigma Z)\right|
\le\frac{A+2}{d-1},$$
thus for $\epsilon$ bounded below as above and also satisfying 
$\epsilon>\frac{2(A+2)}{d-1},$ 
$$\P\left[d_{BL}(W,\sigma Z)>\epsilon
\right]\le\frac{48\sqrt{\pi B}}{
\epsilon^{3/2}}\exp\left[-\frac{
(d-1)\epsilon^5}{9\cdot 2^{16}B^2}\right].$$

\end{proof}

\bigskip

\begin{proof}[Proof of Corollary \ref{wait}]
The proof is essentially trivial.  Note that $$\P[T_\epsilon>m]\ge\left[1-
\frac{c_1\sqrt{B}}{\epsilon^{3/2}}\exp\left(\frac{c_2(d-1)\epsilon^5}{B^2}
\right)\right]^m$$ by independence of the $\theta_j$ and Theorem \ref{main},
since $T_\epsilon>m$ if and only if $d_{BL}(\mu_n^{\theta_j},\n(0,\sigma^2)\le
\epsilon$ for all $1\le j\le m$.  
This bound can be used in the identity $\E T_\epsilon=\sum_{m=0}^\infty\P[
T_\epsilon>m]$ to obtain the bound in the corollary.
\end{proof}

\bigskip

{\bf Remark:} One of the features of the proofs given above is that they
can be generalized to the case of $k$-dimensional projections of the 
$d$-dimensional data vectors $\{x_i\}$, with $k$ fixed or even growing
with $d$.  The proof of the higher-dimensional analog of Theorem 
\ref{testfcn} goes through essentially the same way.  
However, the analog of the proof of Theorem \ref{main} from Theorem 
\ref{testfcn} is rather 
more involved in the multivariate setting and will 
be the subject of a future paper.


\end{document}